\documentclass{amsart}
\usepackage{amsmath,amssymb,amsthm,mathrsfs}

\newtheorem{theorem}{Theorem}[section]

\newtheorem{problem}[theorem]{Problem}

\theoremstyle{definition}

\theoremstyle{remark}
\newtheorem{remark}[theorem]{Remark}

\numberwithin{equation}{section}

\newcommand{\C}{{\mathbb C}}
\newcommand{\D}{{\mathbb D}}
\newcommand{\T}{{\mathbb T}}
\newcommand{\cL}{{\mathcal L}}
\newcommand{\cE}{{\mathcal E}}
\newcommand{\cS}{{\mathcal S}}
\newcommand{\cX}{{\mathcal X}}
\newcommand{\cY}{{\mathcal Y}}
\newcommand{\cO}{{\mathcal O}}
\newcommand{\cH}{\mathcal{H}}

\newcommand{\sbm}[1]{\left[\begin{smallmatrix} #1
                \end{smallmatrix}\right]}

\numberwithin{equation}{section}

\begin{document}

\title[Interpolation in de Branges-Rovnyak spaces]
{Interpolation in de Branges-Rovnyak spaces}

\author{Joseph A. Ball}
\address{Joseph A. Ball, Department of Mathematics,
Virginia Tech, Blacksburg, VA 24061-0123, USA}
\email{ball@math.vt.edu}

\author{Vladimir Bolotnikov}
\address{Vladimir Bolotnikov, Department of Mathematics,
The College of William and Mary,
Williamsburg VA 23187-8795, USA}
\email{vladi@math.wm.edu}

\author{Sanne ter Horst}
\address{Sanne ter Horst, Department of Mathematics, Unit for BMI,
North-West University, Potchefstroom 2531, South Africa}
\email{Sanne.TerHorst@nwu.ac.za}

\begin{abstract}
A general interpolation problem with operator argument is studied
for functions $f$ from the de Branges-Rovnyak space $\cH(s)$ associated
with an analytic function $s$ mapping the open unit disk $\D$ into the
closed unit disk.
The interpolation condition is taken in the
Rosenblum-Rovnyak form $f(A)c=b$ (with a suitable interpretation of
$f(A)c$) for given Hilbert space operator $A$ and
two vectors $b,c$ from the same space.
\end{abstract}


\maketitle


\section{Introduction}
The class of all holomorphic functions mapping the unit disk $\D$
into its closure has played a prominent role in function theory and its
applications beginning with the work of I. Schur \cite{Schur};
following now standard terminology, we refer
to this class of functions as the {\em Schur class} and denote it by
${\mathcal S}$.  Since the first results by mathematicians as
 Schur, Carath\'eodory,
Fej\'er, Nevanlinna and Pick in the early part of the last century
with further elaboration for problems involving matrix- and
operator-valued Schur classes to the present, there has now emerged a
rather complete theory for interpolation problems in the Schur class.
Among several alternative characterizations of the Schur class is
one in terms of positive kernels and associated reproducing kernel
Hilbert spaces:  {\em the function $s \colon {\mathbb D} \to {\mathbb
C}$ is in the Schur class ${\mathcal S}$ if and only if the
associated so-called {\em de Branges-Rovnyak kernel}
\begin{equation}
K_s(z,\zeta)=\frac{1-s(z)\overline{s(\zeta)}}{1-z\bar{\zeta}}
\label{1.1}
\end{equation}
is a positive  kernel and thereby gives rise to a reproducing kernel
Hilbert space ${\mathcal H}(K_{s})$} (see \cite{dbr2}). On the other hand
the kernel \eqref{1.1} being positive is equivalent to the operator
$M_s: f\to sf$ of  multiplication by $s$ to be a contraction on the
Hardy space $H^2$ of square summable power series on $\D$;
then the  general complementation theory applied to the
contractive operator
$M_s: \, H^2\to H^2$ provides the characterization of $\cH(K_{s})$ as the
operator range
$\cH(K_s)={\rm Ran}(I-M_sM^*_s)^{\frac{1}{2}}\subset H^2$  with the lifted
norm
$$
\|(I-M_sM^*_s)^{\frac{1}{2}}f\|_{\cH(K_s)}=\|(I-\pi)f \|_{H^2}
$$
where $\pi$ here is the orthogonal projection onto ${\rm
Ker}(I-M_sM^*_s)^{\frac{1}{2}}$. It follows that $\|f\|_{\cH(K_{s})}\ge
\|f\|_{H^2}$ for every $f\in\cH(K_{s})$ and thus de Branges-Rovnyak spaces
are contractively included in $H^2$. Upon setting
$f=(I-M_sM^*_s)^{\frac{1}{2}}h$ in the last formula we get
\begin{equation}
\| (I - M_{s}M_{s}^{*})h\|_{\cH(K_s)}=\langle (I - M_{s} M_{s}^{*})h, \,
h\rangle_{H^2}.
\label{1.2}
\end{equation}
The purpose of this note is to study an interpolation problem of general
type in Rosenblum-Rovnyak form (see \cite{RR})
\begin{equation}  \label{RRint}
 f(A)c = b
 \end{equation}
where the operator $A$ and the
vectors $c$ and $b$ are given) with the unknown function $f$ in a de
Branges-Rovnyak space ${\mathcal H}(K_{s})$ rather than in the Schur
class ${\mathcal S}$. By way of motivation, we note that if we take
the de Branges-Rovnyak data set $(A,c,b)$ to be of the form
$$
 A = \begin{bmatrix} z_{1} & & \\  & \ddots & \\ & & z_{N}
\end{bmatrix}, \quad c = \begin{bmatrix} 1 \\ \ddots \\ 1
\end{bmatrix}, \quad b = \begin{bmatrix} w_{1} \\ \vdots \\ w_{N}
\end{bmatrix}
$$
for some points $z_{1}, \dots, z_{N}$ in the unit disk and complex
values $w_{1}, \dots, w_{N}$, then the interpolation condition
\eqref{RRint} transcribes to
$$
 f(z_{i}) = w_{i} \text{ for } i = 1, \dots, N,
$$
while, if instead we take $(A,b,c)$ to be of the form
$$
  A = \begin{bmatrix} z_{0} & & & \\ 1 & z_{0} & & \\ & \ddots &
  \ddots \\ & & 1 & z_{0} \end{bmatrix}, \quad
  c = \begin{bmatrix} 1 \\ 0 \\ \vdots \\ 0 \end{bmatrix}, \quad
  b = \begin{bmatrix} w_{0} \\ w_{1} \\ \vdots \\ w_{N-1} \end{bmatrix}
$$
(so $A$ is the lower-triangular Jordan cell of size $N \times N$ with
eigenvalue $z_{0}$ in the unit disk), then the interpolation
condition \eqref{RRint} transcribes to
$$ f^{(j)}(z_{0}) = w_{j} \text{ for } j =0,1, \dots, N-1.
$$
In the final section of the paper we return to these two concrete
examples in a combined form to illustrate the general ideas developed
here.

If $s$ is inner or $s\equiv 0$ so that
$\cH(K_s)$ is contained in $H^2$ isometrically or just equals $H^2$,
the solution can be obtained upon combining the orthogonal
projection theorem and the Beurling-Lax theorem on shift-invariant
subspaces of $H^2$. The general case requires somewhat more delicate
arguments. In Theorem \ref{T:1.2} (the main result of this paper) we
present a parametrization of all
solutions to a general interpolation problem with operator argument
based on isometric multipliers between two de Branges-Rovnyak spaces.
The formulas become even more concrete and explicit in the case of
finite-dimensional data and invertibility of the associated Pick
matrix. In the last section, we illustrate the general formalism by
showing how higher-order multipoint interpolation of classical type fits
into the general scheme.

\section{Statement of main result}

To formulate the interpolation problem of interest to us
we first fix notation and recall some needed definitions. The symbol
$\cL(\cX,\cY)$ stands for the algebra of bounded linear
operators mapping a Hilbert space $\cX$ into another Hilbert space $\cY$,
abbreviated to $\cL(\cX)$ in case $\cX=\cY$. If $T\in\cL(\cX)$ and
$E\in\cL(\cX,\C)$, the pair $(E,T)$ is called {\em  output-stable}
if the associated observability operator
\begin{equation}
\cO_{E,T}: \; x\mapsto E(I-zT)^{-1}x=\sum_{n=0}^\infty z^n ET^nx
\label{1.4}
\end{equation}
maps $\cX$ into $H^2$ and is bounded. The pair is called {\em observable}
if the operator $\cO_{E,T}: \; \cX\to H^2$ is injective. For an
output-stable pair
$(E,T)$, we define the tangential functional calculus $f\mapsto f(T^*)E^*$
on $H^2$ by
\begin{equation}
f(T^*)E^*=\sum_{n\ge 0}f_n T^{* n}E^{*}\quad\mbox{if}\quad
f(z)=\sum_{n\ge 0}f_nz^n.
\label{1.5}
\end{equation}
The computation
$\left\langle \sum_{n\ge 0}f_n T^{* n}E^{*}, \; x\right\rangle_{\cX}=
\sum_{n\ge 0} f_{n}\overline{E T^{n}x}=\langle f, \; \mathcal {O}_{E, T}
x\rangle_{H^2}$
shows that the output-stability of $(E, T)$ is exactly
what is needed to verify that the infinite series in the definition
\eqref{1.5} of $f(T^*)E^*$ converges in the weak topology on $\cX$.
The same computation shows that tangential evaluation
with operator argument amounts to the adjoint of $\cO_{E, T}$:
\begin{equation}
f(T^{*})E^*= \cO_{E, T}^{*}f\quad\mbox{for}\quad f \in H^2.
\label{1.6}
\end{equation}
Evaluation \eqref{1.4} certainly applies to functions from de
Branges-Rovnyak spaces $\cH(K_s)\subset H^2$
and suggests the following interpolation problem.
\begin{problem}
Given $s\in\cS$ and given
$T\in\cL(\cX)$ and $E, {\bf x}\in\cL(\cX,\C)$ so that the pair
$(E,T)$ is output stable and observable, find all functions
$f\in\cH(K_s)$ such that
\begin{equation}
f(T^*)E^*={\bf x}^*.
\label{1.7}
\end{equation}
\label{P:2.1}
\end{problem}
The solvability criterion for Problem \ref{P:2.1} is given in terms of a
positive semidefinite operator $P\in\cL(\cX)$ which we now introduce.
We first apply evaluation \eqref{1.5} to the given
$s\in\cS$ to define  $N\in\cL(\cX,\C)$ by
\begin{equation}
N^*:=s(T^{*})E^*= \cO_{E, T}^{*}s\in\cX.
\label{1.8}
\end{equation}
Then the pair $(N,T)$ is output stable (cf. \cite[Proposition 3.1]{bbieot})
and the observability operator $\cO_{N,T}: \; x\mapsto
N(I-zT)^{-1}x$ equals
\begin{equation}
\cO_{N, T}=M_s^*\cO_{E, T}: \; \cX\to H^2.
\label{1.9}
\end{equation}
Since $M_s: \, H^2\to H^2$ is a contraction, it follows from
\eqref{1.9} that the operator
\begin{equation}
P:=\cO_{E, T}^{*}\cO_{E, T}-\cO_{N, T}^{*}\cO_{N, T}=
\cO_{E, T}^{*}(I-M_s^*M_s)\cO_{N, T}
\label{1.10}
\end{equation}
is positive semidefinite.
Another important property of $P$ (see e.g., \cite{bbieot}) is that it
satisfies the Stein equation
\begin{equation}
P-T^*PT=E^*E-N^*N.
\label{1.11}
\end{equation}
As we will see in Remark \ref{R:2.2} below, the condition
${\bf x}^*\in{\rm Ran}\, P^{\frac{1}{2}}$ is necessary and sufficient for
Problem \ref{P:2.1} to have a solution.
In particular, Problem \ref{P:2.1} always has a solution if $P$ is
strictly positive definite, that is if $P^{-1}\in\cL(\cX)$.
Parametrization of all solutions
of the problem in this nondegenerate case is given in Theorem \ref{T:1.2}
below which is the main result of the paper. To formulate it we need some
auxiliary constructions.

\smallskip

We first observe that if $P$ is boundedly
invertible, then so is the observability gramian $\cO_{E, T}^{*}\cO_{E,
T}$ (see formula \eqref{1.10}) which in turn implies (see e.g.,
\cite{BR}) that the operator $T$ is {\em strongly stable} in the sense
that  $T^n$ converge to zero in the strong operator topology. We also
recall that if $T$ is strongly stable, then the Stein equation
\eqref{1.11} has a unique solution (given of course by formula
\eqref{1.10}). Let $J$ be
the signature matrix given by
\begin{equation}
J=\left[\begin{array}{cr}1 &0\\ 0& -1\end{array}\right]\quad\mbox{and
let}\quad \Theta(z)=\begin{bmatrix}a(z) & b(z) \\ c(z) & d(z)\end{bmatrix}
\label{1.12}
\end{equation}
be a $2\times 2$ matrix function such that for all $z,\zeta\in\D$,
\begin{equation}
\frac{J - \Theta(z)J\Theta(\zeta)^{*}}{1 -z\overline{\zeta}}
=\begin{bmatrix}E \\ N\end{bmatrix}(I - zT)^{-1} P^{-1}(I -
\overline{\zeta}T^{*})^{-1}\begin{bmatrix}E^* & N^*\end{bmatrix}.
\label{1.13}
\end{equation}
The function $\Theta$ is determined by equality \eqref{1.13} uniquely up
to
a constant $J$-unitary factor on the right. One possible choice of $\Theta$
satisfying \eqref{1.13} is
$$
\Theta(z)=D+z\begin{bmatrix}E \\ N\end{bmatrix}(I-zT)^{-1}B
$$
where the operator $\sbm{ B \\ D}\colon \C^2\to
\sbm{\cX \\ \C^2}$ is an injective solution
to the $J$-Cholesky factorization problem
$$
\begin{bmatrix} B \\ D \end{bmatrix}J
\begin{bmatrix} B^{*}&  D^{*} \end{bmatrix} =
\begin{bmatrix} P^{-1} & 0 \\ 0 & J
\end{bmatrix} - \begin{bmatrix} T \\ E \\ N
\end{bmatrix} P^{-1} \begin{bmatrix} T^{*} & E^* & N^*\end{bmatrix}
$$
(such a solution exists due to \eqref{1.11}). If ${\rm
spec}(T)\cap{\mathbb
T}\neq{\mathbb T}$ (which is the case if e.g., $\dim \cX<\infty$), then
a function $\Theta$ satisfying \eqref{1.13} can be taken in the form
\begin{equation}
\Theta(z)=I_2+(z-\mu)\left[\begin{array}{c}E^* \\
N^*\end{array}\right]
(I_\cX-zT^*)^{-1}P^{-1}(\mu I_\cX-T)^{-1}\left[\begin{array}{cc}
E & -N\end{array}\right]
\label{1.14}
\end{equation}
where $\mu$ is an arbitrary point in ${\mathbb T}\setminus {\rm spec}(T)$.
For $\Theta$ of the form \eqref{1.14}, the verification of identity
\eqref{1.13} is straightforward and relies on the Stein identity
\eqref{1.11} only.
It follows from \eqref{1.13} that $\Theta$ is $J$-contractive on $\D$,
i.e., that
$\Theta(z)J\Theta(z)^{*}\le J$ for all $z\in\D$. A much less trivial fact is that
due to strong stability of $T$, the function $\Theta$ is $J$-inner, that is,
the nontangential boundary values $\Theta(t)$ exist for almost all $t\in\T$ and are
$J$-unitary: $\Theta(t)J\Theta(t)^{*}= J$. In particular, $|\det \Theta(t)|=1$.
Every $J$-contractive function $\Theta=\sbm{a & b \\ c & d}$ with $\det\Theta\not\equiv 0$
gives rise to the one-to-one linear fractional transform
\begin{equation}
\cE\mapsto {\bf T}_{\Theta}[\cE]:=\frac{a\cE+b}{c\cE+d}
\label{1.15}
\end{equation}
mapping the Schur class $\cS$ into itself. In case $\Theta$ is a $J$-inner function
satisfying identity \eqref{1.13}, the transform \eqref{1.15}
establishes a one-to-one correspondence between $\cS$ and the set of all Schur class
functions $g$ such that $g(T^{*})E^*=N^*$.  Here we define $g(T^{*})E^*$ according to
definition \eqref{1.5}, using the fact that as sets we have the inclusion $\cS\subset H^2$.
Since the given function $s$ satisfies the latter
condition by definition \eqref{1.8} of $N$, it follows that $s={\bf
T}_{\Theta}[\sigma]$
for some (uniquely determined) function $\sigma\in\cS$ which is recovered from $s$ by
\begin{equation}
\sigma=\frac{ds-b}{a-cs}.
\label{1.16}
\end{equation}
The fact that $\sigma$ of the form \eqref{1.16} belongs to the Schur
class can be interpreted as a general version of the Schwarz-Pick lemma.
The last needed ingredient is  the $\cL(\cX,\C)$-valued function
\begin{equation}
F^s(z)=(E-s(z)N)(I-zT)^{-1}
\label{1.17}
\end{equation}
which also is completely determined from interpolation data. For the
multiplication operator $M_{F^s}$ we have
\begin{equation}
M_{F^s}=\cO_{E,T}-M_s\cO_{N,T}=(I-M_s^*M_s)\cO_{E,T}
\label{1.18}
\end{equation}
where the first equality follows from \eqref{1.17} and definitions of
observability
operators and where the second equality is a consequence of \eqref{1.9}.
Formula \eqref{1.18} and the range characterization of $\cH(K_s)$ imply
that
$M_{F^s}$ maps $\cX$ into $\cH(K_s)$. Furthermore, it follows from
\eqref{1.17}, \eqref{1.2} and \eqref{1.4} that
$$
\|F^sx\|^2_{\cH(K_s)}=
\langle (I - M_s M_s^{*}) \cO_{E,T}x, \cO_{E,T}x\rangle_{H^2}
=\langle (\cO_{E,T}^{*} \cO_{E,T} -
          \cO_{N,T}^{*}\cO_{N,T})x, x \rangle_{\cX}
$$
for every $x\in\cX$ which implies, on account of \eqref{1.10}, that
\begin{equation}
\|F^sx\|^2_{\cH(K_s)}=\langle Px, \, x\rangle_{\cX}\quad\mbox{for all}\quad x\in\cX.
\label{1.19}
\end{equation}
\begin{theorem}
Let us assume that the data set of Problem \ref{P:2.1} is such that
the operator $P$ defined in \eqref{1.10} is strictly positive definite.
Let
$\Theta=\sbm{a & b \\ c & d}$ be a $J$-inner function satisfying
\eqref{1.13} and let $\sigma\in\cS$ and $F^s$ be given as in \eqref{1.16}
and \eqref{1.17}.
Then
\begin{enumerate}
\item All solutions $f$ of
Problem \ref{P:2.1} are parametrized by the formula
\begin{equation}
f=F^sP^{-1}{\bf x}^*+(a-sc)h,
\label{1.20}
\end{equation}
where $h$ is a free parameter from the space $\cH(K_\sigma)$.
\item Representation \eqref{1.20} is orthogonal in the metric of
$\cH(K_s)$.
\item The multiplication operator $M_{a-sc}: \, \cH(K_\sigma)\to \cH(K_s)$ is isometric.
\item For every $f$ of the form \eqref{1.20},
$\, \|f\|^2_{\cH(K_s)}={\bf x}P^{-1}{\bf x}^*+\|h\|^2_{\cH(K_\sigma)}$.
\end{enumerate}
\label{T:1.2}
\end{theorem}
It follows from Theorem \ref{T:1.2} that the norm-constrained version of
Problem \ref{P:2.1} can be solved with no additional efforts:
given $\gamma\ge 0$, all solutions $f$ to the problem  ${\bf
P}(\cH(K_s))$ satisfying additional condition $\|f\|^2_{\cH(K_s)}\le
\gamma$ are parametrized by formula
\eqref{1.20} where the parameter $h$ varies in $\cH(K_\sigma)$  and is
such that $\|h\|^2_{\cH(K_\sigma)}\le \gamma-{\bf x}P^{-1}{\bf x}^*$.

\section{The proof}

In this section we present the proof of Theorem \ref{T:1.2}.
We will make use of multiplication operators $M_{F^s}: \, \cX\to \cH(K_s)$
and $M_f: \, \C\to\cH(K_s)$ for the function $F^s$ defined in \eqref{1.17}
and
for the interpolant $f\in\cH(K_s)$. Since $\cH(K_s)\subset H^2$, the adjoints
of $M_{F^s}$ and $M_f$ can be taken in the metric of $H^2$ as well as in the metric
of $\cH(K_s)$ which are not the same unless $s$ is inner. To avoid confusion
we will use notation $A^{[*]} $ for the adjoint of $A$ in the metric of $\cH(K_s)$.
In terms of these adjoints, we have
\begin{equation}
\|f\|^2_{\cH(K_S)}=M_f^{[*]}M_f\quad\mbox{and}\quad P=M_{F^s}^{[*]}M_{F^s}
\label{2.1}
\end{equation}
where the first equality is self-evident and the second rephrases
\eqref{1.19}.
Furthermore, upon making subsequent use of \eqref{1.18}, \eqref{1.2} and
\eqref{1.6},
we see that for  every $f\in\cH(K_s)$ and $x\in\cX$,
\begin{align*}
\langle x, M_{F^s}^{[*]}f\rangle_\cX&=\langle F^sx, \, f\rangle_{\cH(K_s)}=
\langle (I-M_sM_s^*)\cO_{E,T}x, \, f\rangle_{\cH(K_s)}\\
&=\langle \cO_{E,T}x, \, f\rangle_{H^2}=\langle x,
\, \cO_{E,T}^*f\rangle_{\cX}=\langle x, f(T^{*})E^*\rangle_{\cX}
\end{align*}
which shows that interpolation condition \eqref{1.7} can be written
as
\begin{equation}
M_{F^s}^{[*]}f={\bf x}^*.
\label{2.2}
\end{equation}
\begin{remark}
Problem \ref{P:2.1} has a solution if and only if ${\bf
x}^*\in{\rm Ran}\, P^{\frac{1}{2}}$.
\label{R:2.2}
\end{remark}
\begin{proof}[\bf Proof.] By the second equality in \eqref{2.1}, ${\rm Ran}\,
P^{\frac{1}{2}}={\rm Ran} \, M_{F^s}^{[*]}$. Thus, ${\bf x}^*$ belongs to
${\rm Ran}\, P^{\frac{1}{2}}$ if and only if it belongs to ${\rm Ran} \,
M_{F^s}^{[*]}$, that is, if and only if equality \eqref{2.2} holds for
some $f\in\cH(K_s)$ which means that this $f$ solves Problem
\ref{P:2.1}.\end{proof}

The next theorem characterizes solutions to Problem \ref{P:2.1}
in terms of positive kernels. Characterizations of this type
were first applied to interpolation problems by V. Potapov \cite{kopot}.
\begin{theorem}
Given $\gamma>0$, a function $f: \, \D\to\C$ is a solution of
Problem \ref{P:2.1}
and satisfies the norm constraint $\|f\|_{\cH(K_s)}^2\le \gamma$ if and only if
the following kernel is positive on $\D\times\D$:
\begin{equation}
{\bf K}(z,\zeta)=\begin{bmatrix}\gamma & {\bf x} & f(\zeta)^*\\
{\bf x}^* & P & F^s(\zeta)^{*} \\ f(z) & F^s(z) & K_s(z,\zeta)
\end{bmatrix}\succeq 0\qquad (z,\zeta\in\D),
\label{2.3}
\end{equation}
where $P$ and $F^s$ are defined in \eqref{1.10} and \eqref{1.4}.
\label{T:2.1}
\end{theorem}

\begin{proof}[\bf Proof.]  Let us assume that $f\in\cH(K_s)$ satisfies
$\|f\|_{\cH(K_s)}^2\le
\gamma$ and meets interpolation condition \eqref{1.7} (or equivalently,
condition
\eqref{2.2}). The operator ${\mathbb P}\in\cL(\C\oplus\cX\oplus\cH(K_s))$ defined
below is positive semidefinite:
\begin{equation}
{\mathbb P}=\begin{bmatrix}M_f^{[*]} \\ M_{F^s}^{[*]} \\
I_{\cH(K_s)}\end{bmatrix}\begin{bmatrix}M_f & M_{F^s} & I_{\cH(K_s)}
\end{bmatrix}\ge 0.
\label{2.4}
\end{equation}
Taking this into account, along with equalities \eqref{2.1} and \eqref{2.2}, we then have
\begin{equation}
0\le {\mathbb P}=\begin{bmatrix}\|f\|_{\cH(K_s)}^2 & {\bf x} & M_f^{[*]} \\
{\bf x}^* & P &  M_{F^s}^{[*]} \\ M_f &  M_{F^s} & I_{\cH(K_s)}
\end{bmatrix}\le \begin{bmatrix}\gamma & {\bf x} & M_f^{[*]} \\
{\bf x}^* & P &  M_{F^s}^{[*]} \\ M_f &  M_{F^s} & I_{\cH(K_s)}
\end{bmatrix}=:{\bf P}.
\label{2.5}
\end{equation}
We next observe for every vector $g\in\C\oplus{\mathcal X}\oplus \cH(K_S)$ of
the form
\begin{equation}
g(z)=\sum_{j=1}^r\left[\begin{array}{c}c_j \\ x_j\\
K_s(\cdot,z_j)y_j\end{array}\right]\qquad
(c_j, y_j\in\C, \; x_j\in{\mathcal X}, \; z_j\in\D).
\label{2.7}
\end{equation}
the identity
\begin{equation}
\left\langle {\bf P}g, \; g\right\rangle_{\C\oplus\cX\oplus
\cH(K_s)}=\sum_{j,\ell=1}^r \left\langle{\bf K}(z_j,z_\ell)
\sbm{c_\ell \\ x_\ell \\ y_\ell}, \; \sbm{c_j \\ x_j \\ y_j}
\right\rangle_{\C\oplus{\mathcal X}\oplus \C}
\label{2.6}
\end{equation}
holds. Since ${\bf P}\ge 0$, the quadratic form on the  right hand side
of \eqref{2.6} is nonnegative, which proves \eqref{2.3}.

\smallskip

Conversely, let us assume
that the kernel \eqref{2.3} is positive on $\D\times\D$. Since the set of
vectors of the form \eqref{2.7} is dense in $\C\oplus{\mathcal X}\oplus
\cH(K_S)$, the identity \eqref{2.6} now implies that the operator
${\bf P}$ is positive semidefinite and therefore, the
Schur complement of the $(3,3)$-block in ${\bf P}$ is positive
semidefinite. On account of \eqref{2.5}, we have
$$
\begin{bmatrix}\gamma-M_f^{[*]}M_f & {\bf x}-M_f^{[*]}M_{F^s} \\
{\bf x}^*-M_{F^s}^{[*]}M_f & P-M_{F^s}^{[*]} M_{F^s}\end{bmatrix}=
\begin{bmatrix}\gamma-M_f^{[*]}M_f & {\bf x}-M_f^{[*]}M_{F^s} \\
{\bf x}^*-M_{F^s}^{[*]}M_f & 0\end{bmatrix}\ge 0
$$
which implies \eqref{2.2} (and therefore \eqref{1.7}) and
$\|f\|^2_{\cH(K_s)}=M_f^{[*]}M_f\le \gamma$  which completes
the proof.\end{proof}

\begin{proof}[\bf Proof of Theorem \ref{T:1.2}.]
We now assume that $P$ is
strictly positive definite.

\smallskip

{\em Proof of (1):} We first assume that $f$ is a solution of Problem
\ref{P:2.1}. Then the kernel \eqref{2.3} (with
$\gamma=\|f\|_{\cH(K_s)}^2$)
is positive by Theorem \ref{T:2.1}. Since  $P$ is boundedly invertible,
we can take its Schur complement in ${\bf K}$ to get the equivalent inequality
\begin{equation}
\begin{bmatrix} \gamma-{\bf x} P^{-1}{\bf x}^* &
f(\zeta)^*-{\bf x}P^{-1}F^s(\zeta)^*\\
f(z)-F^s(z)P^{-1}{\bf x} & K_s(z,\zeta)-F^s(z)P^{-1}F^s(\zeta)^*
\end{bmatrix}\succeq 0.
\label{2.8}
\end{equation}
We first observe that in view of  \eqref{1.1}, \eqref{1.17} and
\eqref{1.13}, the $(2,2)$-block in \eqref{2.8} can be written in the form
\begin{equation}
K_s(z,\zeta)-F^s(z)P^{-1}F^s(\zeta)^*=
\begin{bmatrix}1 & -s(z)\end{bmatrix}
\frac{\Theta(z)J\Theta(\zeta)^*}{1-z\bar{\zeta}}\begin{bmatrix}1
\\ \\  -\overline{s(\zeta)}\end{bmatrix}.
\label{2.9}
\end{equation}
We next observe that for $\sigma\in\cS$ defined as in \eqref{1.16},
\begin{equation}
\begin{bmatrix}1 &
-s\end{bmatrix}\Theta=\begin{bmatrix}a-cs & b-ds\end{bmatrix}=
u\begin{bmatrix}1 & -\sigma\end{bmatrix}\quad\mbox{where}\quad
u:=a-cs.
\label{2.10}
\end{equation}
Substituting \eqref{2.10} into \eqref{2.9} gives
\begin{equation}
K_s(z,\zeta)-F^s(z)P^{-1}F^s(\zeta)^*=u(z)
\frac{1-\sigma(z)\overline{\sigma(\zeta)}}{1-z\bar{\zeta}}u(\zeta)^*
=u(z)K_\sigma(z,\zeta)\overline{u(\zeta)}
\label{2.11}
\end{equation}
which in turn allows us to write \eqref{2.8} as
\begin{equation}
\begin{bmatrix} \gamma-{\bf x} P^{-1}{\bf x}^* &
\overline{f(\zeta)}-{\bf x}P^{-1}F^s(\zeta)^*\\
f(z)-F^s(z)P^{-1}{\bf x} & u(z)K_\sigma(z,\zeta)\overline{u(\zeta)}
\end{bmatrix}\succeq 0.
\label{2.12}
\end{equation}
By \cite[Theorem 2.2]{beabur}, the latter inequality means that
the function
$$
h:=u^{-1}(f-F^sP^{-1}{\bf x})
$$
belongs to the space $\cH(K_\sigma)$ with norm $\|h\|_{\cH(K_\sigma)}\le (\gamma-{\bf x}
P^{-1}{\bf x}^*)^{\frac{1}{2}}$. The  desired representation \eqref{1.20}
for $f$ now follows from definitions of $h$ and $u$.

\smallskip

Conversely, if we start with a function $h\in\cH(K_\sigma)$ with norm
$\|h\|_{\cH(K_\sigma)}=(\gamma-{\bf x}P^{-1}{\bf x}^*)^{\frac{1}{2}}$
for some $\gamma\ge {\bf x}P^{-1}{\bf x}^*$ and define $f$ by formula
\eqref{1.20},
then again by \cite[Theorem 2.2]{beabur} we have
$$
\begin{bmatrix}\gamma-{\bf x} P^{-1}{\bf x}^* &
\overline{h(\zeta)u(\zeta)}\\ u(z)h(z) & u(z)K_\sigma(z,\zeta)\overline{u(\zeta)}
\end{bmatrix}\succeq 0
$$
which is the same as \eqref{2.12} and also the same as \eqref{2.8} (due to
\eqref{2.11}).
Since $P$ is invertible, inequality \eqref{2.8} is equivalent to \eqref{2.3} which implies
(by Theorem \ref{T:2.1}) that $f$ solves Problem \ref{P:2.1}.

\medskip

{\em Proof of (2):} Upon letting $h\equiv 0$ in \eqref{1.20} we conclude
that
the function $F^sP^{-1}{\bf x}^*$ is a particular solution to  Problem
\ref{P:2.1}. Therefore, the second term on the right side of
\eqref{1.20} is a general solution to the homogeneous problem:
\begin{equation}
((a-cs)h)(T^*)E^*=\cO_{E,T}^*(a-cs)h=M_{F^s}^{[*]}(a-cs)h=0\quad\mbox{for all}\quad
h\in\cH(K_\sigma).
\label{2.13}
\end{equation}
For every $h\in\cH(K_\sigma)$ we now have
$$
\langle (a-cs)h, \, F^sP^{-1}{\bf x}^*\rangle_{\cH(K_s)}=
\langle M_{F^s}^{[*]}(a-cs)h, \, P^{-1}{\bf x}^*\rangle_{\cX}=0
$$
which completes the proof of (2).

\medskip

{\em Proof of (3):} As in part (1) we use notation $u:=a-cs$. Take a function
\begin{equation}
h(z)=\sum_{i=1}^k \alpha_i
K_\sigma(z,z_i)\overline{u(z_i)}, \quad \alpha_i\in\C; \; z_i\in\D, \; k\ge 0.
\label{2.14}
\end{equation}
This function belongs to
$\cH(K_\sigma)$ and by the reproducing kernel property,
\begin{equation}
\|h\|^2_{\cH(K_\sigma)}=\sum_{i,j=1}^k
\alpha_i\overline{\alpha}_ju(z_i)\overline{u(z_j)}K_\sigma(z_i,z_j).
\label{2.15}
\end{equation}
We will show that $\|uh\|_{\cH(K_s)}=\|h\|_{\cH(K_\sigma)}$. Since the set of all
functions of the form \eqref{2.4} is dense in $\cH(K_\sigma)$ (recall that $u\not\equiv 0$),
the desired statement will follow. To complete the proof use formula
\eqref{2.11} to get
\begin{equation}
u(z)h(z)=\sum_{i=1}^k \alpha_i
u(z)K_\sigma(z,z_i)\overline{u(z_i)}=\sum_{i=1}^k\alpha_i
(K_s(z,z_i)-F^s(z)P^{-1}F^s(z_i)^*).
\label{2.16}
\end{equation}
By the reproducing property and by \eqref{1.19},
\begin{align*}
\langle K_s(\cdot,z_j),  K_s(\cdot,z_i)\rangle_{\cH(K_s)}&=K_s(z_i,z_j),\\
\langle F^sP^{-1}F^s(z_j)^*,  K_s(\cdot,z_i)\rangle_{\cH(K_s)}&=F^s(z_i)P^{-1}F^s(z_j)^*,\\
\langle F^sP^{-1}F^s(z_j)^*,
F^sP^{-1}F^s(z_i)^*\rangle_{\cH(K_s)}&=F^s(z_i)P^{-1}F^s(z_j)^*,
\end{align*}
which together with \eqref{2.16} leads us to
\begin{equation}
\|uh\|_{\cH(K_s)}^2=\sum_{i,j=1}^k \alpha_i\overline{\alpha}_j
\left(K_s(z_i,z_j)-F^s(z_i)P^{-1}F^s(z_j)^*\right).
\label{2.17}
\end{equation}
Since by \eqref{2.11},
$\; K_s(z_i,z_j)-F^s(z_i)P^{-1}F^s(z_j)^*=
u(z_i)\overline{u(z_j)}K_\sigma(z_i,z_j), \; $
it follows that  the right hand side expressions in \eqref{2.17} and \eqref{2.15} are equal.
Thus, $\|uh\|_{\cH(K_s)}=\|h\|_{\cH(K_\sigma)}$, which completes the proof of (3).

\medskip

{\em Proof of (4):} This part follows from parts (2), (3) and the equality
$$\|F^sP^{-1}{\bf x}^*\|^2_{\cH(K_s)}=
{\bf x}P^{-1}{\bf x}^*$$
which in turn, is a consequence of \eqref{1.19}.
\end{proof}

\section{An example}

Parametrization \eqref{1.20} is especially transparent in case
$\dim \cX<\infty$. Then with respect to an appropriate basis of $\cX$
the output stable observable pair $(E,T)$ has following form: $T$ is a block
diagonal matrix $T={\rm diag}\{T_1,\ldots,T_k\}$ with the diagonal block
$T_i$ equal the upper triangular $n_i\times n_i$ Jordan block with the
number $\overline{z}_i\in\D$ on the main diagonal and $E$ is the row
vector
$$
E=\begin{bmatrix}E_1 & \ldots & E_k\end{bmatrix},\quad \mbox{where}\quad
E_i=\begin{bmatrix}1 & 0 & \ldots 0\end{bmatrix}\in\C^{1\times n_i}.
$$
It is not hard to show that for $(E,T)$ as above and for every function
$f$ analytic at $z_1,\ldots,z_k$, evaluation \eqref{1.5} amounts to
\begin{equation}
f(T^*)E^*=\operatorname{Col}_{1\le i\le k}\operatorname{Col}_{0\le j<
n_i}f^{(j)}(z_i)/j!.
\label{3.1}
\end{equation}
If we  specify the entries of the column ${\bf x}^*$ by letting
$$
{\bf x}^*=\operatorname{Col}_{1\le i\le k}\operatorname{Col}_{0\le j<
n_i} x_{ij},
$$
then it is readily seen that Problem \ref{P:2.1} amounts
to the following Lagranges-Sylvester interpolation problem:

\medskip

${\bf LSP}$: {\em Given $s\in\cS$,
$k$ distinct points $z_1,\ldots,z_k\in\D$
and a collection $\left\{x_{ij}\right\}$ of complex numbers, find all
functions $f\in\cH(K_s)$ such that
$$
f^{(j)}(z_i)/j!=x_{ij}\quad\mbox{for}\quad j=0,\ldots,n_i-1; \;
i=1,\ldots,k.
$$}
The auxiliary column $N^*$ is now defined from the derivatives of the
given function $s$ via formula \eqref{3.1} and we define the matrix $P$ as
the unique solution of the Stein equation \eqref{1.11}. This matrix $P$
turns out to be equal to the Schwarz-Pick matrix
$$
P=\left[\left[\left.\frac{1}{\ell !r!} \,
\frac{\partial^{\ell+r}}{\partial
z^\ell\partial\bar{\zeta}^r} \,
\frac{1-s(z)\overline{s(\zeta)}}{1-z\bar{\zeta}}
\right\vert_{{\scriptsize\begin{smallmatrix} z=z_i\\
\zeta=z_j\end{smallmatrix}}}
\right]_{\ell=0,\ldots,n_i-1}^{r=0,\ldots,n_j-1}\right]_{i,j=1}^k
$$
which in turn is known to be positive definite  unless $s$
is a Blaschke product of degree $m<n:=n_1+\ldots+n_k$ in
which case $P$ is positive semidefinite and ${\rm rank} \, P=m$.
Thus the case $P>0$ handled in Theorem \ref{T:1.2} is generic
if $\dim \cX<\infty$. In this generic case, all the solutions $f$ to the
problem ${\bf LSP}$ are given by formula \eqref{1.20}
where now all the ingredients are not only explicit but also computable;
for example $\Theta$ is a rational $J$-inner function which can be taken
in the form \eqref{1.14}.  In particular, the unique solution of
minimal norm $f_{\text{min}}$ has the form
$$
f_{\text{min}}(z) = (E - s(z) N) (I - zT)^{-1} P^{-1} x^{*}
$$
which is quite explicit in terms of the given data, at least given
that one i willing to compute the inverse $P^{-1}$ of the
Schwarz-Pick matrix $P$.

\smallskip

If $s$ is a Blaschke product of degree $m<n$, the problem ${\bf
LSP}$ has a unique solution. Here is the sketch of the
proof. Let $m_1,\ldots,m_k$ be any nonnegative integers such that
$m_1+\ldots+m_k=m=\deg s$ and $m_i\le n_i$ for $i=1,\ldots,k$. Let us
consider the problem $\widetilde{\bf LSP}$ with interpolation
conditions
$$
f^{(j)}(z_i)/j!=x_{ij}\quad\mbox{for}\quad j=0,\ldots,m_i-1; \;
i=1,\ldots,k.
$$
Let $\widetilde{\bf x}$, $\widetilde{E}$, $\widetilde{N}$,
$\widetilde{T}$ and $\widetilde{P}$
be the matrices associated with this truncated problem. The $m\times m$
matrix $\widetilde{P}$ is a principal submatrix of $P$ and it is invertible
since $\deg s=m$. Let $\widetilde{\Theta}=\sbm{\widetilde{a} &
\widetilde{b} \\ \widetilde{c} & \widetilde{d}}$ be constructed by
formula
\eqref{1.14} (with $\widetilde{E}$, $\widetilde{N}$, $\widetilde{T}$ and
$\widetilde{P}$ instead of $E$, $N$, $T$ and $P$ respectively). It
then turns out that the function
$\widetilde\sigma=\frac{\widetilde{d}s-\widetilde{b}}
{\widetilde{a}-\widetilde{c}s}$ is the Blaschke product of degree zero,
that is a unimodular constant. By Theorem \ref{T:1.2}, all solutions of
the
truncated problem $\widetilde{\bf LSP}$ are of the form
$$
f=\widetilde{F}^s\widetilde{P}^{-1}\widetilde{\bf x}^*+(a-sc)h,
\quad\mbox{where}\quad \widetilde{F}^s=(\widetilde{E}-s(z)\widetilde{N})
(I-z\widetilde{T})^{-1}
$$
and where $h$ is a function from $\cH(K_\sigma)$. Since $\sigma$ is a
unimodular constant, the space $\cH(K_\sigma)$ is trivial and thus
the problem $\widetilde{\bf LSP}$ has only one solution
$\widetilde{F}^s\widetilde{P}^{-1}\widetilde{\bf x}^*$.
It finally can be shown that if the necessary condition ${\bf x}^*\in{\rm
Ran} \, P^{\frac{1}{2}}$ holds, then this function also solves the
larger problem ${\bf LSP}$.

\smallskip

The case where $P$ is not boundedly invertible is much more
challenging and if $\dim \cX=\infty$ and especially if the
function $s$ belongs to the {\em operator-valued} Schur class in which
case the space $\cH(K_s)$ consists of vector-valued functions
(rather than scalar-valued).
This case will be worked out on a separate occasion.

\bibliographystyle{amsplain}

\begin{thebibliography}{10}


\bibitem{bbieot}
J.~A.~Ball and V.~Bolotnikov, {\it Interpolation problems for Schur
multipliers on the
Drury-Arveson space:  from Nevanlinna-Pick to Abstract Interpolation
Problem}, Integral Equations Operator Theory {\bf 62} (2008), no. 3,
301-349.

\bibitem{bgr} J.~A.~Ball, I.~Gohberg, and L.~Rodman.
\emph{Interpolation of rational matrix functions},
OT45, Birkh\"{a}user Verlag, 1990.

\bibitem{BR} J.A.~Ball and M.W.~Raney, Discrete-time dichotomous
well-posed linear systems and generalized Schur-Nevanlinna-Pick
interpolation, {\em Complex Anal.~Oper.~Theory} \textbf{1}
(2007), 1--54.

\bibitem{beabur}
F.~Beatrous and J.~Burbea,
\emph{Positive-definiteness and its applications to interpolation
problems for holomorphic functions},
Trans. Amer. Math. Soc., \textbf{284} (1984), no.1, 247--270.


\bibitem{dbr2}
{L. de} Branges and J.~Rovnyak,
{\emph Square summable power series},
Holt, Rinehart and Winston, New--York, 1966.




\bibitem{kopot}
I.~V.~Kovalishina and V.~P.~Potapov, {\em Seven Papers Translated from
the Russian}, American Mathematical Society Translations {\bf 138},
Providence, R.I., 1988.


\bibitem{RR} M. Rosenblum and J. Rovnyak,
\emph{Hardy Classes and Operator Theory}, Oxford University Press, 1985.


\bibitem{Schur}
I.~Schur, \emph{\"Uber Potenzreihen, die im Innern des
Einheitskreises beschr\"ankt sind. I}, J. Reine Angew. Math. {\bf 147}
(1917), 205­-232.

\end{thebibliography}
\providecommand{\bysame}{\leavevmode\hbox to3em{\hrulefill}\thinspace}

\end{document}